\documentclass[12pt,leqno]{article}
\usepackage{amssymb,amsmath,amsfonts,amsbsy, amsthm, xspace, latexsym, amscd, graphicx, fancyhdr}
\usepackage[colorlinks=true, pdfstartview=FitV, linkcolor=blue, citecolor=blue, urlcolor=blue]{hyperref}

\swapnumbers
\theoremstyle{definition}
\newtheorem{definition}{Definition}[section]
\newtheorem{remark}[definition]{Remark}
\newtheorem{example}[definition]{Example}

\theoremstyle{plain}
\newtheorem{theorem}[definition]{Theorem}

\newtheorem{lemma}[definition]{Lemma}

\newcommand{\ifff}{{\em i\,f\,{f}\;  }}

\newcommand{\supp}{\ensuremath{{\rm{supp}}}}

\newcommand{\card}{\ensuremath{{\rm{card}}}}

\newcommand{\dom}{\ensuremath{{\rm{dom}}}}

\newcommand{\st}{\ensuremath{{{\rm{st}}}}}

\newcommand{\rest}{\!\!\ensuremath{\upharpoonright}\!}

\newcommand{\Diff}{\ensuremath{{\rm{Diff}}}}

\newcommand{\eqae}{\ensuremath{ \overset{\centerdot}{=}}}

\newcommand{\lae}{\ensuremath{\overset{\centerdot}{<}}}

\newcommand{\gae}{\ensuremath{\overset{\centerdot}{>}}}

\newcommand{\leqae}{\ensuremath{\overset{\centerdot}{\leq}}}

\begin{document}
\title{An axiomatic approach to the non-linear theory of generalized functions and consistency of Laplace transforms}
\author{Todor D. Todorov\\ 
                        Mathematics Department\\                
                        California Polytechnic State University\\
                        San Luis Obispo, California 93407, USA\\
																							(ttodorov@calpoly.edu)}
\date{}
\maketitle
\begin{abstract} We offer an axiomatic definition of a differential algebra of generalized functions over an algebraically closed non-Archimedean field. This algebra is of {\em Colombeau type} in the sense that it contains a copy of the space of Schwartz distributions. We study the uniqueness of the objects we define and the consistency of our axioms. Next, we identify an inconsistency in the conventional Laplace transform theory. As an application we offer a free of contradictions alternative in the framework of our algebra of generalized functions. The article is aimed at mathematicians, physicists and engineers who are interested in the non-linear theory of generalized functions, but who are not necessarily familiar with the original Colombeau theory. We assume, however, some basic familiarity with the Schwartz theory of distributions.
\end{abstract}



	 *Partly supported by START-project Y237 of the Austrian Science Fund.
\section{Introduction}

	We define a field of generalized numbers $\widehat{\mathbb{C}}$ and an algebra of generalized functions $\widehat{\mathcal{E}(\Omega)}$ over $\widehat{\mathbb{C}}$ by means of several axioms. We show that these axioms determine $\widehat{\mathbb{C}}$ uniquely up to a field isomorphism. We prove that our axioms are consistent by showing that the field of generalized numbers and the algebra of generalized functions constructed in (\cite{TodVern08}, \S1-5) offers a model for our axioms. The algebra 
$\widehat{\mathcal{E}(\Omega)}$ is of {\em Colombeau type} in the sense that it contains a copy of the space of Schwartz distributions. However, the ring $\overline{\mathbb{C}}$ of the original Colombeau generalized numbers (see ~\cite{jfCol84a}) does not satisfy our axioms, because $\overline{\mathbb{C}}$ is a ring with zero divisors, in contrast to $\widehat{\mathbb{C}}$, which is an algebraically closed Cantor complete field. We should mention that the field of generalized numbers and the algebra of generalized functions constructed in \cite{OberTod98} also presents a model for the axioms in Section~\ref{S: Generalized Scalars and Functions in Axioms} provided that the non-standard extension $^*\mathbb{R}$ of $\mathbb{R}$ used in \cite{OberTod98} is fully-saturated or, more generally, a special model (see \cite{TodWolf04}, \S 7), and also   $\card(^*\mathbb{R})=\mathfrak{c}^+$, where $\mathfrak{c}^+$ is the successor of $\mathfrak{c}=\card(\mathbb{R})$.

	Most of our axioms are algebraic in nature. Others - such as the axiom about the ring of the $\mathcal{C}^\infty$-functions (from an open set to a Cantor complete field) - are borrowed from analysis. Because of the common and traditional nature of our framework, we believe that our {\em axiomatic approach} might be useful to mathematicians from different areas of mathematics who would like to grasp at least the basic ideas of Colombeau theory without being involved from the very beginning in the technical details of Colombeau's construction. The author of this article has repeatedly tested this axiomatic approach in communicating with colleagues from different areas of pure and applied mathematics without preliminary knowledge on the subject - both on the blackboard and on a piece of paper (and even on a napkin over a glass of wine).  Still we assume that the typical reader of this article is familiar with Schwartz's theory of distributions and, more importantly, has an appreciation for its usefulness in science.  

	At the end of the article we identify an inconsistency in the conventional Laplace transform theory. As an application we offer a free of contradictions alternative in the framework of the algebra of generalized functions $\widehat{\mathcal{E}(\Omega)}$.

	Let $\Omega$ be an open subset of $\mathbb{R}^d$. In what follows we denote by $\mathcal{E}(\Omega)=:\mathcal{C}^\infty(\Omega)$ the space of $\mathcal{C}^\infty$-functions from $\Omega$ to $\mathbb{C}$ and by $\mathcal{D}(\Omega)=:\mathcal{C}_0^\infty(\Omega)$ - the space of test functions on $\Omega$. We shall often use $\mathcal{D}$ for $\mathcal{D}(\mathbb{R}^d)$ for short. We denote by $\mathcal{D}^\prime(\Omega)$ the space of Schwartz distributions on $\Omega$ and by  $\mathcal{E}^\prime(\Omega)$ - the space of the distributions in $\mathcal{D}^\prime(\Omega)$ with compact support \cite{vVladimirov}. Similarly, we denote by $\mathcal{S}(\Omega)$ and $\mathcal{S}^\prime(\Omega)$ the Schwartz space of rapidly decreasing functions and the space of tempered distributions on $\Omega$, respectively (\cite{mGrosser at al 2}, p. 26). We denote by $\mathcal{T}^d$ the \textbf{usual topology} on $\mathbb{R}^d$. Most of the algebraic terms used in the article can be found in (\cite{van der Waerden}, Ch.11). For topics related to valuation fields we refer to the Introduction in \cite{TodWolf04}, where the reader will find more references to the subject.
\section{Generalized scalars and functions in axioms}\label{S: Generalized Scalars and Functions in Axioms}
We describe a field of generalized numbers $\widehat{\mathbb{R}}$,  its complex companion $\widehat{\mathbb{C}}=\widehat{\mathbb{R}}(i)$ and an algebra of generalized functions $\widehat{\mathcal{E}(\Omega)}$ over the field $\widehat{\mathbb{C}}$ of Colombeau type by means of several axioms. The consistency of these axioms will be discussed later in this article.  

\begin{enumerate}
\item[{\bf Axiom 1}] (Transfer Principle). $\widehat{\mathbb{R}}$ is a \textbf{real closed field} (\cite{van der Waerden}, 11.5).
\item[{\bf Axiom 2}] (First Extension Principle). $\widehat{\mathbb{R}}$ contains $\mathbb{R}$ as a proper subfield, i.e. $\mathbb{R}\subsetneqq\widehat{\mathbb{R}}$.
\end{enumerate}
The designation {\em Transfer Principle} of Axiom 1 is due to the fact that all real closed fields - in particular $\mathbb{R}$ and $\widehat{\mathbb{R}}$ - are indistinguishable under the first order formal language (which does not involve the cardinality of completeness of the fields). For more details on this topic we refer to \cite{MarkerMessPill}.
\begin{theorem}(Non-Archimedean Field).\label{T: Non-Archimedean Field} $\widehat{\mathbb{R}}$ is \textbf{orderable} in a unique way by: $x\geq 0$ in $\widehat{\mathbb{R}}$ if $x=y^2$ for some $y\in\widehat{\mathbb{R}}$. Consequently, $\widehat{\mathbb{R}}$ is a \textbf{non-Archimedean totally ordered field} (i.e. $\widehat{\mathbb{R}}$ has non-zero infinitesimals). If $x\in\widehat{\mathbb{R}}$ is an infinitesimal (i.e. $|x|< 1/n$ for all $n\in\mathbb{N}$), we shall often write $x\approx 0$ for short. Also, $\widehat{\mathbb{R}}$ is a \textbf{topological field} under the order topology on $\widehat{\mathbb{R}}$. Consequently, $\widehat{\mathbb{R}}^d$ is a \textbf{topological vector space} under the product-order topology inherited from the order topology on $\widehat{\mathbb{R}}$.
\end{theorem}
\begin{proof} We refer the reader to  (\cite{van der Waerden}, 11.5, Theorem 1, p.249).
\end{proof}

	In what follows $\mathfrak{c}^+$ stands for the successor of $\mathfrak{c}=\card(\mathbb{R})$.
\begin{enumerate}
\item[{\bf Axiom 3}] (Completeness Principle). $\widehat{\mathbb{R}}$ is \textbf{Cantor $\mathfrak{c}^+$-complete} in the sense that every family $\{[a_\gamma, b_\gamma]\}_{\gamma\in\Gamma}$ of closed intervals in $\widehat{\mathbb{R}}$ with the finite intersection property and $\card(\Gamma)\leq\mathfrak{c}$ has a non-empty intersection $\bigcap_{\gamma\in\Gamma} [a_\gamma, b_\gamma]\not=\varnothing$. 

\item[{\bf Axiom 4}] (Cardinality Principle). $\card(\widehat{\mathbb{R}})=\mathfrak{c}^+$.
\end{enumerate}

\begin{remark}\label{R: Cardinality of Cantor Complete Non-Archimedean Field} Axiom 4 can be replaced by a slightly weaker, let us call it  Axiom $4^\prime$: $\card(\widehat{\mathbb{R}})\leq \mathfrak{c}^+$, because Axiom 3 implies $\card(\widehat{\mathbb{R}})\geq\mathfrak{c}^+$. Indeed, let $\mathcal{I}(\widehat{\mathbb{R}}_+)$ denote the set of all positive infinitesimals in $\widehat{\mathbb{R}}$. We observe that $\mathcal{I}(\widehat{\mathbb{R}}_+)$ is non-empty, since $\widehat{\mathbb{R}}$ is non-Archimedean by Theorem~\ref{T: Non-Archimedean Field}. Next, we observe that the family $\{[a, b]\}_{a\in\mathcal{I}(\widehat{\mathbb{R}}_+),\ b\in{\widehat{\mathbb{R}}_+\setminus\mathcal{I}(\widehat{\mathbb{R}}_+)}}$ has the finite intersection property, but its intersection is empty. Also, we observe that 
$\card\Big(\mathcal{I}(\widehat{\mathbb{R}}_+)\times{(\widehat{\mathbb{R}}_+\setminus\mathcal{I}(\widehat{\mathbb{R}}_+))}\Big)\leq\card(\widehat{\mathbb{R}}\times\widehat{\mathbb{R}})=\card(\widehat{\mathbb{R}})$. Thus it follows $\card(\widehat{\mathbb{R}})\geq\mathfrak{c}^+$ (as required), since $\widehat{\mathbb{R}}$ is Cantor $\mathfrak{c}^+$-complete by Axiom 3. Still we prefer our (slightly stronger) Axiom 4 (over Axiom 4$^\prime$) for the sake of simplicity. 
\end{remark}

\begin{enumerate}
\item[{\bf Axiom 5}] (Existence of Scale). $\widehat{\mathbb{R}}$ contains an \textbf{infinitesimal scale}, i.e. there exists $s\in\widehat{\mathbb{R}}$ such that: (a) $(\forall n\in\mathbb{N})(0< s <1/n$; (b) the sequence of intervals $(-s^n, s^n)$ in $\widehat{\mathbb{R}}$ forms a base for the neighborhoods of the zero in the interval topology on $\widehat{\mathbb{R}}$. We shall keep $s$ fixed in what follows.
\item[{\bf Axiom 6}] (Exponentiation). $\widehat{\mathbb{R}}$ admits {\bf exponentiation} in the sense that there exists a strictly decreasing function $\exp_s: \mathcal{F}(\widehat{\mathbb{R}})\to\widehat{\mathbb{R}}_+$ which is a group isomorphism between $(\mathcal{F}(\widehat{\mathbb{R}}), +)$ and $(\widehat{\mathbb{R}}_+, \cdot)$ such that $(\forall q\in\mathbb{Q})(\exp_s(q)=s^q)$. We shall often write $s^x$ instead of $\exp_s(x)$.
\end{enumerate}

	Notice that the exponents $s^q$ are well defined in $\widehat{\mathbb{R}}$ for all $q\in\mathbb{Q}$ since $s\in\widehat{\mathbb{R}}_+$ and $\widehat{\mathbb{R}}$ is a real closed field by Axiom 1. In what follows $\mathcal{F}(\widehat{\mathbb{R}})$ denotes the ring of finite elements $x$ of $\widehat{\mathbb{R}}$, i.e. for which $|x|\leq n$ for some $n\in\mathbb{N}$.  Notice that the inverse $\log_s: \widehat{\mathbb{R}}_+\to\mathcal{F}(\widehat{\mathbb{R}})$ of $\exp_s$ exists and $\ln{s}=1/\log_s{e}$.

\begin{definition} (Valuation). We define a valuation $v:\widehat{\mathbb{R}}\to\mathbb{R}\cup\{\infty\}$ (depending on $s$) by $v(0)=\infty$ and $v(x)=\sup\{q\in\mathbb{Q}: |x|/s^q\approx 0\}$ if $x\not=0$.
\end{definition}
\begin{theorem}(Properties of Valuation). $v$ is a \textbf{non-Archimedean valuation} which agrees with the order on $\widehat{\mathbb{R}}$ in the sense that $(\forall x, y\in\widehat{\mathbb{R}})$: (a) $v(x)=\infty$ \ifff $x=0$; (b) $v(xy)=v(x)+v(y)$; (c) $v(x+y)\geq\min\{v(x), v(y)\}$; (d) $(|x|<|y|\implies v(x)\geq v(y))$ {\em (\cite{LiRob}, Ch.1, \S4)}. 
\end{theorem}
\begin{proof} We leave the verification to the reader.
\end{proof}

\begin{theorem}(Algebraic and Topological Properties).\label{T: Algebraic and Topological Properties}

 \begin{enumerate}
\item[(i)] $\widehat{\mathbb{C}}=:\widehat{\mathbb{R}}(i)$ is an \textbf{algebraically closed field}. Also, $\widehat{\mathbb{C}}$ is a \textbf{valuation field} under the valuation inherited from $\widehat{\mathbb{R}}$ by means of the formula $v(z)=v(|z|)$. 

\item[(ii)] $\widehat{\mathbb{C}}$ is \textbf{spherically complete ultra-metric space} under the \textbf{valuation metric} $d_v(x, y)=e^{-v(x-y)}$ in the sense that every nested sequence of closed balls in $\widehat{\mathbb{C}}$ has a non-empty intersection. Consequently, both $\widehat{\mathbb{R}}$ and $\widehat{\mathbb{C}}$ are sequentially complete. 

\item[(iii)] The product-order topology and the metric topology (sharp topology) coincide on $\widehat{\mathbb{C}}$. 
\item[(iv)] Let $(a_n)$ be a sequence in $\widehat{\mathbb{C}}$. Then $\lim_{n\to\infty}a_n=0$  \ifff $\lim_{n\to\infty}v(a_n)=\infty$   \ifff $\sum_{n=0}^\infty a_n$ is convergent in $\widehat{\mathbb{C}}$.
\end{enumerate}
\end{theorem}
\begin{proof} (i) $\widehat{\mathbb{C}}$ is an algebraically closed field by the Artin-Schreier theorem (\cite{van der Waerden}, 11.5, Theorem 3a, p.251), since 
$\widehat{\mathbb{R}}$ is a real closed field by Axiom 1. For the properties of the valuation metric we refer to (\cite{LiRob}, Ch.1, \S5).
	
	(ii) The sequential completeness of $\widehat{\mathbb{R}}$ (hence, of $\widehat{\mathbb{C}}$) follows from the Cantor $\mathfrak{c}^+$-completeness (Axiom 4). 

	(iii) follows from Axiom 5.  

	(iv) holds for any complete ultra-metric space (\cite{LiRob}, p. 21-22).
\end{proof}

\begin{definition}(Infinitesimal Relation in $\widehat{\mathbb{C}}$).\label{D: Infinitesimal Relation} Let $z\in\widehat{\mathbb{C}}$. We say that $z$ is \textbf{infinitesimal}, in symbol $z\approx 0$, if $|z|<1/n$ for all $n\in\mathbb{N}$. Similarly, $z$ is called \textbf{finite} if $|z|\leq n$ for some $n\in\mathbb{N}$. Finally, $z$ is called \textbf{infinitely large} if $n<|z|$ for all $n\in\mathbb{N}$. If $S\subset\widehat{\mathbb{C}}$, we denote by $\mathcal{I}(S)$, $\mathcal{F}(S)$ and $\mathcal{L}(S)$ the sets of the infinitesimal, finite and infinitely large numbers in $S$, respectively. The infinitesimal relation in $\widehat{\mathbb{C}}$ is defined as follows: if $z, t\in\widehat{\mathbb{C}}$, we write $z\approx t$ if $z-t\in\mathcal{I}(\widehat{\mathbb{C}})$. In particular, $z\approx 0$ \ifff $z\in\mathcal{I}(\widehat{\mathbb{C}})$.
\end{definition}

\begin{lemma}(Standard Part Mapping in $\widehat{\mathbb{C}}$).\label{L: Standard Part Mapping in Chat} We have $\mathcal{F}(\widehat{\mathbb{C}})=\mathbb{C}\oplus \mathcal{I}(\widehat{\mathbb{C}})$ in the sense that every finite $z\in\widehat{\mathbb{C}}$ has a unique ``asymptotic expansion'' $z=c+dz$ where $c\in\mathbb{C}$ and $dz\in\mathcal{I}(\widehat{\mathbb{C}})$. The mapping $\st: \mathcal{F}(\widehat{\mathbb{C}})\to\mathbb{C}$, defined by $\st(c+dz)=c$, is called \textbf{standard part mapping} in $\widehat{\mathbb{C}}$.
\end{lemma}

\begin{proof} The existence of the expansion $z=c+dz$ follows from the completeness of $\mathbb{R}$ and its uniqueness follows from the fact that $\mathbb{C}$ is an Archimedean field.
\end{proof}

\begin{definition}(Monad).\label{D: Monad} Let $\Omega$ be an open subset of $\mathbb{R}^d$. We define the monad of $\Omega$ in $\widehat{\mathbb{R}}^d$ by  $
\widehat{\mu}(\Omega)=\big\{\omega+h:
\omega\in\Omega, ||h||\approx 0, h\in\widehat{\mathbb{R}}^d\big\}$.
\end{definition}

\begin{theorem} $\widehat{\mu}(\Omega)$ is an open subset of $\widehat{\mathbb{R}}^d$ (under the product-order topology on $\widehat{\mathbb{R}}^d$ inherited from the order topology of $\widehat{\mathbb{R}}$).
\end{theorem}
\begin{proof} Let $\varepsilon$ be a positive infinitesimal in $\widehat{\mathbb{R}}$ and let $B_\varepsilon(\omega+h)$ be the open ball in $\widehat{\mathbb{R}}^d$ of radius $\varepsilon$, centered at $\omega+h$. Then $B_\varepsilon(\omega+h)\subset\widehat{\mu}(\Omega)$.
\end{proof}

\begin{definition} We denote by $\mathcal{C}^\infty(\widehat{\mu}(\Omega), \widehat{\mathbb{C}})$ the ring of the $\mathcal{C}^\infty$-functions from $\widehat{\mu}(\Omega)$ to $\widehat{\mathbb{C}}$ (i.e. $\mathcal{C}^\infty(\widehat{\mu}(\Omega), \widehat{\mathbb{C}})$ consists of all functions from $\widehat{\mu}(\Omega)$ to $\widehat{\mathbb{C}}$ whose  iterated partial derivatives exist on $\widehat{\mu}(\Omega)$).
\end{definition}

\begin{enumerate}
\item[{\bf Axiom 7}] (Standard Embedding).   There exists an embedding (an injective mapping) 
$\sigma_\Omega:\mathcal{E}(\Omega)\to\mathcal{C}^\infty(\widehat{\mu}(\Omega), \widehat{\mathbb{C}})$ which preserves the differential ring operations in $\mathcal{E}(\Omega)$ and such that for every $f\in\mathcal{E}(\Omega)$ the function $\sigma(f)$ is a (pointwise) extension of $f$, i.e. 
$\sigma_\Omega(f)\rest\Omega=f$, where\,  $\rest\Omega$\,  stands for the pointwise restriction on $\Omega$. We shall call $\sigma_\Omega$ {\em standard embedding}. 

\item[{\bf Axiom 8}]  (Second Extension Principle). There exists a differential subalgebra $\widehat{\mathcal{E}(\Omega)}$ of $\mathcal{C}^\infty(\widehat{\mu}(\Omega), \widehat{\mathbb{C}})$ over the field $\widehat{\mathbb{C}}$ supplied with a linear pairing $(\cdot|\cdot)$ between 
$\widehat{\mathcal{E}(\Omega)}$ and $\mathcal{D}(\Omega)$ with values in $\widehat{\mathbb{C}}$ such that $\sigma_\Omega[\mathcal{E}(\Omega)]\subsetneqq\widehat{\mathcal{E}(\Omega)}$. 
\end{enumerate}

\begin{remark}(Missing Axiom).\label{R: Internal Functions}  We feel that in some informal (and still unclear) sense the $\widehat{\mathcal{E}(\Omega)}$ consists of the ``internal'' elements of $\mathcal{C}^\infty(\widehat{\mu}(\Omega), \widehat{\mathbb{C}})$, where ``internal'' is used in the spirit of \cite{OberVern08}. One unsolved for now problem in our axiomatic approach is how to characterize $\widehat{\mathcal{E}(\Omega)}$ (uniquely) as a particular subset of  $\mathcal{C}^\infty(\widehat{\mu}(\Omega), \widehat{\mathbb{C}})$ in terms of the Axioms 1-8 and possibly, some additional, unknown to us, axioms.
\end{remark}

\begin{enumerate}
\item[{\bf Axiom 9}](Colombeau Embedding). There exists an \textbf{(Colombeau type of) embedding} (an injective mapping) $E_\Omega:\mathcal{D}^\prime(\Omega)\to\widehat{\mathcal{E}(\Omega)}$  such that: 
\begin{enumerate}

\item[(a)] $E_\Omega$ preserves the linear operations in $\mathcal{D}^\prime(\Omega)$ (including the partial differentiation of any order); 
\item[(b)] $E_\Omega$ preserves the usual pairing between 
$\mathcal{D}^\prime(\Omega)$ and $\mathcal{D}(\Omega)$, between $\mathcal{S}^\prime(\Omega)$ and $\mathcal{S}(\Omega)$, and between $\mathcal{E}^\prime(\Omega)$ and $\mathcal{E}(\Omega)$ in the sense that $\big(T\big|\tau\big)=\big(E_\Omega(T)\big|\tau\big)$ for all $T\in\mathcal{D}^\prime(\Omega)$ and all $\tau\in\mathcal{D}(\Omega)$ and similarly for the other two pairs.

\item[(c)] $E_\Omega$ agrees with  $\sigma_\Omega$ in the sense that $(E_\Omega\circ L_\Omega)(f)=\sigma_\Omega(f)$, for all $f\in\mathcal{E}(\Omega)$, where $L_\Omega:\mathcal{L}_{loc}(\Omega)\to\mathcal{D}^\prime(\Omega)$ is the Schwartz embedding defined by the formula $\big(L_\Omega(f)\big| \tau\big)=\int_\Omega f(x)\tau(x)\, dx$ (\cite{vVladimirov}, \S 1.6, p.18).

\item[(d)] If $T$ is a real distribution (\cite{vVladimirov}, p.12), then $E_\Omega(T)$ is a real-valued function in the sense that $(\forall x\in\widehat{\mu}(\Omega))(E_\Omega(T)(x)\in\widehat{\mathbb{R}})$.

\item[(e)] $E_\Omega(\delta_\lambda)(x)=0$ for all $\lambda\in\Omega$ and all $x\in\widehat{\mathbb{R}}^d$ such that $||x-\lambda||\geq s$. Here $\delta_\lambda\in\mathcal{D}^\prime(\Omega)$ denote the Dirac distribution (Dirac delta function) with $\supp(\delta_\lambda)=\{\lambda\}$ (which is commonly written as  $\delta_\lambda=\delta(x-\lambda)$) and $s$ is the scale of $\widehat{\mathbb{R}}$ (Axiom 5).
\end{enumerate}
\end{enumerate}

	We \textbf{summarize} all of these in the chain of inclusions 
$\mathcal{E}(\Omega)\subset\mathcal{D}^\prime(\Omega)\subset
\widehat{\mathcal{E}(\Omega)}\subset\mathcal{C}^\infty(\widehat{\mu}(\Omega), \widehat{\mathbb{C}})$,
 where $\mathcal{E}(\Omega)$ is a \textbf{differential
subalgebra} of $\widehat{\mathcal{E}(\Omega)}$ over $\mathbb{C}$, $\mathcal{D}^\prime(\Omega)$  is a \textbf{differential
vector subspace} of $\widehat{\mathcal{E}(\Omega)}$ over $\mathbb{C}$ and $\widehat{\mathcal{E}(\Omega)}$ is a differential subalgebra of $\mathcal{C}^\infty(\widehat{\mu}(\Omega), \widehat{\mathbb{C}})$ over  $\widehat{\mathbb{C}}$.

	In what follows we denote by $\mathcal{T}^d$ the \textbf{usual topology} on $\mathbb{R}^d$.

\begin{definition}(Restriction). For every $\mathcal{O}, \Omega\in\mathcal{T}^d$ with $\mathcal{O}\subseteq \Omega$, we define ${\rm res}_{\mathcal{O},\Omega}: \widehat{\mathcal{E}}(\Omega)\to\widehat{\mathcal{E}}(\mathcal{O})$ by  ${\rm res}_{\mathcal{O},\Omega}(f)=f\rest\widehat{\mu}(\mathcal{O})$, where\,  $\rest\widehat{\mu}(\mathcal{O})$ stands for the pointwise restriction on $\widehat{\mu}(\mathcal{O})$ (Definition~\ref{D: Monad}).
\end{definition}

\begin{enumerate}	
\item [{\bf Axiom 10}](Sheaf Principle). (a) Each ${\rm res}_{\mathcal{O},\Omega}$ is a homomorphism of differential algebras over the field $\widehat{\mathbb{C}}$; (b) The family $\big\{\widehat{\mathcal{E}(\Omega)}\big\}_{\Omega\in\mathcal{T}^d}$ is a \textbf{sheaf of differential algebras} over the field $\widehat{\mathbb{C}}$ under ${\rm res}_{\mathcal{O},\Omega}$ (see~\cite{aKaneko}, \S2); (c) The embeddings $\sigma_\Omega$ and $E_\Omega$ are both sheaf-preserving in the sense that 
$(\forall f\in\mathcal{E}(\Omega))[{\rm res}_{\mathcal{O},\Omega}(\sigma_\Omega(f))=\sigma_\mathcal{O}(f\rest\mathcal{O})]$ and $(\forall T\in\mathcal{D}^\prime(\Omega))[{\rm res}_{\mathcal{O},\Omega}(E_\Omega(T))=E_\mathcal{O}(T\rest\mathcal{O})]$, where\, $\rest\mathcal{O}$ in the latter formula stands for the restriction on $\mathcal{O}$ in the sense of the theory of distributions (\cite{vVladimirov}, p.16-18). 
\end{enumerate}

\begin{definition}(External and Internal Support).\label{D: External and Internal Support} Let $f\in\widehat{\mathcal{E}}(\Omega)$. 
\begin{enumerate}
\item  The \textbf{external support} or simply, the \textbf{support}  $\supp(f)$ of $f$ is the complement to $\Omega$ of the largest open subset $\mathcal{O}$ of $\Omega$ such that ${\rm res}_{\mathcal{O},\Omega}(f)=0$. 

\item The {\bf internal support} ${\rm Supp}(f)$ of $f$ is the the closure (in the order-product topology of $\widehat{\mathbb{R}}^d$) of the set $\{x\in\widehat{\mathbb{R}}^d : f(x)\not=0\}$.
\end{enumerate}
\end{definition}
	We leave the proof of the next lemma to the reader. 
\begin{lemma}(Preservation of External Support).\label{L: Preservation of External Support} 

\begin{enumerate}

\item[(i)] The embedding $E_\Omega$ preserves the external support.

\item [(ii)] Let $f\in\widehat{\mathcal{E}}(\Omega)$. Then $\supp(f)=\{x\in\mathbb{R}: x\approx \xi \text{\; for some\; } \xi\in{\rm Supp}(f)\}.$
\end{enumerate}

\end{lemma}
\begin{definition}(Integral).\label{D: Integral} Let $f\in\widehat{\mathcal{E}}(\Omega)$ and $\supp(f)$ be a compact subset of $\Omega$. Then we define $\int_\Omega f(x)\, dx\in\widehat{\mathbb{C}}$ by $\int_\Omega f(x)\, dx=(f|\gamma)$, where $(\cdot|\cdot)$ is the pairing mentioned in Axiom 8 and $\gamma\in\mathcal{E}(\Omega)$ is a smooth function which is equal to $1$ on a neighborhood of $\supp(T)$. 
\end{definition}

\begin{theorem}(Some Properties of the Integral).\label{T: Integral and Pairing} 
\begin{enumerate}
\item[(i)] For every $f\in\widehat{\mathcal{E}}(\Omega)$ and every test function $\tau\in \mathcal{D}(\Omega)$ we have \newline$\int_\Omega f(x)\sigma(\tau)(x)\, dx=(f|\tau)$. 
\item[(ii)] The embedding $E_\Omega$ preserves the integral in the space $\mathcal{E}^\prime(\Omega)$ in the sense that for every $f\in\mathcal{E}^\prime(\Omega)$ and every $\tau\in \mathcal{E}(\Omega)$ we have $\int_\Omega E_\Omega(f)(x)\sigma(\tau)(x)\, dx\newline=(f|\tau)$. 
\end{enumerate}
\end{theorem}
\begin{proof} The result follows directly from Axiom 10.
\end{proof}

\begin{remark}(Notation). Let $x\in\mu(\Omega)$ and $f\in\mathcal{E}(\Omega)$. We shall sometimes write simply $f(x)$ instead of the more precise $\sigma_\Omega(f)(x)$, e.g. $e^s$ means $\sigma(e^x)(s)$. Similarly, if $T\in\mathcal{D}^\prime(\Omega)$ is a Schwartz distributions, we shall often write simply $T(x)$ instead of the more precise $E_\Omega(T)(x)$ if no confusion could arise. In this notation we have  $\int_\Omega T(x)\tau(x)\, dx=(T|\tau)$ for every distribution $T\in\mathcal{D}^\prime(\Omega)$ and every test function $\tau\in \mathcal{D}(\Omega)$ (notation used by the physicists and engineers anyway).
\end{remark}

\begin{definition}(Weak Equalities in $\widehat{\mathcal{E}}(\Omega)$).\label{D: Weak Equalities} Let $f, g\in\widehat{\mathcal{E}}(\Omega)$. 

\begin{enumerate}
\item We say that $f$ and $g$ are \textbf{weakly equal}, in symbol $f\cong g$, if $(f|\tau)=(g|\tau)$ for all test functions $\tau\in \mathcal{D}(\Omega)$.

\item We say that $f$ and $g$ are \textbf{infinitely close} or \textbf{associated}, in symbol $f\sim g$, if $(f|\tau)\approx (g|\tau)$ for all test functions $\tau\in \mathcal{D}(\Omega)$, where $\approx$ is the infinitesimal relation in $\widehat{\mathbb{C}}$ (Definition~\ref{D: Infinitesimal Relation}).

\item We define ${\rm Mon}(\mathcal{D}^\prime(\Omega))=\{f\in\widehat{\mathcal{E}(\Omega)}: f\sim T \text{\; for some\;} T\in\mathcal{D}^\prime(\Omega)\}$ and also ${\rm Mon}(0)=\{f\in\widehat{\mathcal{E}(\Omega)}: f\sim 0\}$, where $0$ in the latter formula stands for the ``zero-distribution''. 
\end{enumerate}

	We observe that either of $\cong$ or $\sim$ reduces to the usual equality, $=$, on $\mathcal{D}^\prime(\Omega)$ (if $\mathcal{D}^\prime(\Omega)$  is treated as a subset of $\widehat{\mathcal{E}}(\Omega)$). In addition, we have $\sigma_\Omega(\psi)E_\Omega(T)\cong E_\Omega(\psi T)$ for all $\psi\in\mathcal{E}(\Omega)$ and all $T\in\mathcal{D}^\prime(\Omega)$, where the product $\psi T$ is in the sense of distribution theory \cite{vVladimirov}.
\end{definition}
\begin{enumerate}	
\item [{\bf Axiom 11}] (Diffeomorphism Principle).\label{A: Diffeomorphisms} The family $\big\{\widehat{\mathcal{E}(\Omega)}\big\}_{\Omega\in\mathcal{T}^d}$ is  {\bf weakly} {\bf diffeo-morphism}{\bf-invariant} in the sense that for every $\Omega,\mathcal{O} \in\mathcal{T}^d$, every
$T\in\mathcal{D}^\prime(\Omega)$ and every
$\psi\in\Diff(\mathcal{O}, \Omega)$ we have  $E_\Omega(T)\circ\psi
\cong E_{\mathcal{O}}(T\circ\psi)$, i.e. $\left(E_\Omega(T)\circ\psi
\left.\right|\tau\right)=\left(E_{\mathcal{O}}(T\circ\psi)
\left.\right|\tau\right)$ for all test functions $\tau\in\mathcal{D}(\Omega)$. 
\end{enumerate}	

\begin{enumerate}	
\item [{\bf Axiom 12}] (Infinitesimal Translations).\label{A: Infinitesimal Translations} $\widehat{\mathcal{E}(\Omega)}$ is  {\bf closed under infinitesimal translations} in the sense that for every $f\in\widehat{\mathcal{E}(\Omega)}$ and every $h\in\mathcal{I}(\widehat{\mathbb{R}}^d)$ we have $f_h\in\widehat{\mathcal{E}(\Omega)}$. Here $f_h:\widehat{\mu}(\Omega)\to\widehat{\mathbb{C}}$ stands for $f_h(x)=f(x-h)$. 
\end{enumerate}	
\begin{enumerate}	
\item [{\bf Axiom 13}] (Projection Principle).\label{A: Projection Principle} Let $\Lambda$ and $\Omega$ be open sets of $\mathbb{R}^p$ and $\mathbb{R}^d$, respectivley, and let $f\in\widehat{\mathcal{E}(\Lambda\times\Omega)}$. Then $(\forall\lambda\in\widehat{\mu}(\Lambda))(f(\lambda,\, \cdot\,)\in\widehat{\mathcal{E}(\Omega)}$).
\end{enumerate}	
\begin{example}\label{Ex: Exponents} Let $f: \mathbb{R}^3\to \mathbb{C}$ be the function defined by $f(x, y, t)=e^{-(x+iy)t}$. We have $\sigma_{\mathbb{R}^3}(f)\in\widehat{\mathcal{E}(\mathbb{R}^3)}$ by Axiom 7. Also, for every $ x, y \in\widehat{\mu}(\mathbb{R})$ we have $f(x, y, \, \cdot\,)$ $\in\widehat{\mathcal{E}(\mathbb{R})}$ by Axiom 13. We shall often write $e^{-(x+iy)t}$ or even $e^{-zt}$ (where $z=x+iy$)  instead of the more preicise $\sigma_{\mathbb{R}^3}(e^{-(x+iy)t})$ or $\sigma_{\mathbb{R}^3}(e^{-zt})$, respectively.
\end{example}
\begin{example} Let $\lambda\in\Omega$ and let $\delta_\lambda\in\mathcal{D}^\prime(\Omega)$ be the Dirac delta distribution with $\supp(\delta_\lambda)=\{\lambda\}$ and $E_\Omega(\delta_\lambda)$ be its image in $\widehat{\mathcal{E}}(\Omega)$.  Notice that the powers $\delta_\lambda^n$ does not make sense in Schwartz's theory of distributions for $n=2,3,\dots$, while the powers $(E_\Omega(\delta_\lambda))^n$ are well defined since $\widehat{\mathcal{E}}(\Omega)$ is an algebra. In what follows we shall write simply $\delta(x-\lambda)$ instead of $E_\Omega(\delta_\lambda)$ and $\delta^n(x-\lambda)$ instead of $(E_\Omega(\delta_\lambda))^n$. Notice that ${\rm Supp}(\delta(x-\lambda))=\{x\in\widehat{\mathbb{R}}^d: ||x-\lambda||\leq s\}$, where $s$ is the scale of $\widehat{\mathbb{R}}$ (Axiom 5). If $\psi\in\mathcal{E}(\Omega)$, then $\psi(x)\delta(x-\lambda)\cong \psi(\lambda)\delta(x-\lambda)$. Similarly, we introduce the Heaviside step-function $H$ and the products $\delta^n\delta^{(n)}, H^n,\delta H$, etc. In particular, $\supp(H)=\{x\in\mathbb{R}: x\geq 0\}$
and  ${\rm Supp}(H)=\{x\in\widehat{\mathbb{R}}: x\geq -s\}$. Thus $H(x)=0$ for all $x\in\widehat{\mathbb{R}}, x\leq -s$ and $H(x)=1$ for all $x\in\widehat{\mathbb{R}}, x\geq s$. We have $H^n\sim H$ and $H\delta\sim\frac{1}{2}\delta$. 
\end{example}
\begin{example}. Let $||\delta||_2=\Big(\int_{-\infty}^\infty \delta^2(x)\, dx\Big)^{1/2}$ be the $L_2$-norm of the delta function. Notice that $||\delta||_2\in\widehat{\mathbb{R}}$ and $||\delta||_2\not=0$ (Definition~\ref{D: Integral}). We define the \textbf{normalized Dirac function} $
\Delta_\lambda(x)=\frac{1}{||\delta||_2}\delta(x-\lambda)$. Thus $||\Delta_\lambda(x)||_2=1$. 
Also, $(x\Delta_\lambda(x)|\tau(x))=(\lambda\Delta_\lambda(x)|\tau(x)$ for all $\tau\in\mathcal{D}(\Omega)$ by (b) of Axiom 9. Notice that $
\Delta_\lambda(x)$ is without a counterpart in distribution theory.
\end{example}

\begin{example}\label{Ex: Shifted Delta}. Next, we consider the generalized function $\delta(x-2s)\in \widehat{\mathcal{E}(\mathbb{R}^d)}$. Here $\delta(x-2s)$ is the infinitesimal translation of $\delta(x)$, i.e. $\delta(x-2s)=\delta_{2s}(x)$ (Axiom 12). We observe that $\delta(x-2s)$ is not a distribution, because $s$ is a non-zero infinitesimal (Axiom 5). Also, $\int_{-\infty}^\infty\delta(x-2s)\tau(x)\, dx =\tau(2s)$ for all $\tau\in\mathcal{E}(\mathbb{R}^d)$. Thus $\int_{-\infty}^\infty\delta(x-2s)\tau(x)\, dx \approx \tau(0)$. Consequently, if $\psi\in\mathcal{E}(\Omega)$, then $\psi(x)\delta(x-2s)\cong \psi(2s)\delta(x-2s)\cong \psi(0)\delta(x-2s)$. Also, for every $z\in\mathbb{C}$ we have $\int_{-\infty}^\infty H(x)\delta(x-2s)e^{-zx}\, dx= e^{-2sz}$.
\end{example}

\section{Uniqueness of $\widehat{\mathbb{R}}$ and $\widehat{\mathbb{C}}$}

	We show that Axiom 1-6 determines uniquely $\widehat{\mathbb{R}}$ and $\widehat{\mathbb{C}}$ up to a field isomorphism.
\begin{theorem} If there exists a field $\widehat{\mathbb{R}}$ satisfying Axiom 1-6, then $\widehat{\mathbb{R}}$ is unique up to a field isomorphism which preserves the scale. Consequently, the algebraically closed field $\widehat{\mathbb{C}}=:\widehat{\mathbb{R}}(i)$ is also uniquely determined by Axiom 1-6 up to a field isomorphism.
\end{theorem}

	A proof of this result appears in \cite{TodWolf04} which is written in the framework of Robinson's non-standard analysis. We present below a short ``translation of this proof into standard language''. We should warn about some notational differences: the counterparts of
$\widehat{\mathbb{R}}, \mathcal{F}_v, \mathcal{I}_v, \widehat{\mathcal{F}_v}, \mathbb{M}, \mathbb{F}$ and $ \mathbb{K}((t^\mathbb{R}))$ in this article are denoted in \cite{TodWolf04} by: $^\rho\mathbb{R}, \mathcal{F}_\rho, \mathcal{I}_\rho, \widehat{^\rho\mathbb{R}}, \widehat{\mathbb{R}}, {^*\mathbb{R}}$ and $\mathbb{K}(t^\mathbb{R})$, respectively.
\begin{proof}
\begin{enumerate}
\item We define $\mathcal{F}_v=\big\{x\in\widehat{\mathbb{R}}: v(x)\geq 0\big\}$ and $\mathcal{I}_v=\big\{x\in\widehat{\mathbb{R}}: v(x)>0\big\}$,
and let $\widehat{\mathcal{F}_v}=\mathcal{F}_v/\mathcal{I}_v$. We observe that $\mathcal{F}_v=\big\{x\in\widehat{\mathbb{R}}: (\forall n\in\mathbb{N})(|x|<1/\sqrt[n]{s})\big\}$ and $\mathcal{I}_v=\big\{x\in\widehat{\mathbb{R}}: (\exists n\in\mathbb{N})(|x|\leq\sqrt[n]{s})\big\}$.

\item  \label{I: Maximal Field} Let $\mathbb{M}$ be a subfield of $\widehat{\mathbb{R}}$, containing $\mathbb{R}$, which is maximal in $\mathcal{F}_v$. Then $\mathcal{F}_v=\mathbb{M}\oplus\mathcal{I}_v$  in the sense that $(\forall x\in\mathcal{F}_v)(\exists!r\in\mathbb{M})(x-r\in\mathcal{I}_v)$. Consequently, $\mathbb{M}$ is a \textbf{field of representatives} for $\widehat{\mathcal{F}_v}$ in the sense that $\widehat{\mathcal{F}_v}\equiv\mathbb{M}$ (\cite{ZariskiSamuel}, p. 281). Notice that $v(x)=0$ for all $x\in\mathbb{M}\setminus\{0\}$. 

\item \label{I: Residue Field} (i) $\widehat{\mathcal{F}_v}$ is a real closed non-Archimedean field (\cite{TodWolf04}, p. 363); (ii) $\widehat{\mathcal{F}_v}$ is algebraically $\mathfrak{c}^+$-saturated in the sense that every family $\{(a_\gamma, b_\gamma)\}_{\gamma\in\Gamma}$ of open intervals in $\widehat{\mathcal{F}_v}$ with the finite intersection property and $\card(\Gamma)\leq\mathfrak{c}$ has a non-empty intersection $\bigcap_{\gamma\in\Gamma} (a_\gamma, b_\gamma)\not=\varnothing$ (\cite{TodWolf04}, Theorem 7, p. 370); (iii) $\card(\widehat{\mathcal{F}_v})=\mathfrak{c}^+$. Indeed, (ii) implies
$\card(\widehat{\mathcal{F}_v})\geq \mathfrak{c}^+$, since $\widehat{\mathcal{F}_v}$ is Cantor complete non-Archimedean field (Remark~\ref{R: Cardinality of Cantor Complete Non-Archimedean Field}). On the other hand, \#\ref{I: Maximal Field} implies $\card(\widehat{\mathcal{F}_v})\leq\mathfrak{c}^+$ since $\card(\widehat{\mathbb{R}})=\mathfrak{c}^+$ by Axiom 3. The reader will observe that (i)-(iii) follow only from Axiom 1-6.

\item Let $\mathbb{F}$ be a field satisfying (i)-(iii) in \#\ref{I: Residue Field} (instead of $\widehat{\mathcal{F}_v}$). Then $\mathbb{F}\equiv\widehat{\mathcal{F}_v}$ (\cite{TodWolf04}, \S 7, p. 369-371).

\item Let $\mathbb{K}$ be a field. We denote by $\mathbb{K}\left<t^\mathbb{R}\right>$ the set of the \textbf{Levi-Civita~\cite{tLC} series} with coefficients in $\mathbb{K}$, i.e. the power series of the form $\sum_{n=0}^\infty a_n t^{\nu_n}$, where $a_n\in\mathbb{K}$, $(\nu_n)$ is a sequence in $\mathbb{R}$ such that $\nu_0<\nu_1<\nu_2<\dots$, $\lim_{n\to\infty}\nu_n=\infty$ and $t$ is an indeterminate. We denote by $\mathbb{K}((t^\mathbb{R}))$ the field of \textbf{Hahn power series} $\sum_{r\in\mathbb{R}}a_r t^r$ with coefficients $a_r$ in $\mathbb{K}$, where $\{r\in\mathbb{R}: a_r\not=0\}$ is a well-ordered set (see~\cite{hH}). If $\mathbb{K}$ is a real closed field, then both $\mathbb{K}\left<t^\mathbb{R}\right>$ and $\mathbb{K}((t^\mathbb{R}))$ are also real closed fields. Also, by a result due to Krull~\cite{wKrull}, $\mathbb{K}((t^\mathbb{R}))$ is a maximal immediate extension of $\mathbb{K}\left<t^\mathbb{R}\right>$. 

\item We define $\mathbb{M}\left<s^\mathbb{R}\right>=:\big\{\sum_{n=0}^\infty a_n s^{\nu_n}: \sum_{n=0}^\infty a_n t^{\nu_n}\in\mathbb{M}\left<t^\mathbb{R}\right>\big\}$. We observe that $\mathbb{M}\left< s^\mathbb{R}\right>\subset\widehat{\mathbb{R}}$, because $s^{\nu_n}$ is well-defined by Axiom 6 and the series $\sum_{n=0}^\infty a_n s^{\nu_n}$ are all convergent in $\widehat{\mathbb{R}}$ by the (iii)-part of Theorem~\ref{T: Algebraic and Topological Properties} since $v(a_n s^{\nu_n})=v(a_n)+v(s^{\nu_n})=v(s^{\nu_n})=\nu_n \to\infty$ as $n\to\infty$. We define $J: \mathbb{M}\left<s^\mathbb{R}\right>\to\mathbb{M}\left<t^\mathbb{R}\right>$ by $J(\sum_{n=0}^\infty a_n s^{\nu_n})=\sum_{n=0}^\infty a_n t^{\nu_n}$.

\item We observe that $J$ is a field isomorphism  such that $J(s)=t$ and $J|\mathbb{M}=id$. Thus, $J^{-1}$ is a field embedding of $\mathbb{M}\left<t^\mathbb{R}\right>$ into $\widehat{\mathbb{R}}$ and, consequently,  $\mathbb{M}\left<s^\mathbb{R}\right>$ is a subfield of $\widehat{\mathbb{R}}$. Also, by a result due to Luxemburg~\cite{wLux}, $\widehat{\mathbb{R}}$ is a maximal immediate extension of  $\mathbb{M}\left<s^\mathbb{R}\right>$ since $\widehat{\mathbb{R}}$ is spherically complete by part (ii) of Theorem~\ref{T: Algebraic and Topological Properties}. 

\item By a theorem due to Kaplansky~\cite{iKa} (characteristic $0$ case of Theorem 5) there exists a field isomorphism $\widehat{J}:\widehat{\mathbb{R}}\to\mathbb{M}((t^\mathbb{R}))$ which is an extension of $J$. Consequently, 
$\widehat{\mathbb{R}}\equiv\mathbb{M}((t^\mathbb{R}))$ (see~\cite{TodWolf04}, \S 6, p.368). 

\item Let $\mathbb{F}$ be a field satisfying (i)-(iii) in \#\ref{I: Residue Field} (instead of $\widehat{\mathcal{F}_v}$). Then $\mathbb{M}\equiv\widehat{\mathcal{F}_v}\equiv\mathbb{F}$ by \#2 and \#4. Thus $\widehat{\mathbb{R}}\equiv\mathbb{F}((t^\mathbb{R}))$ by \#8. This completes the proof, because the properties (i)-(iii) in \#\ref{I: Residue Field} of the field $\mathbb{F}$ follow only from Axiom 1-6.
\end{enumerate} 
\end{proof}
\section{Consistency of the axioms of $\widehat{\mathcal{E}(\Omega)}$}

\begin{theorem} Axiom 1-10 in Section~\ref{S: Generalized Scalars and Functions in Axioms} are consistent (under ZFC (Zermelo-Fraenkel plus Axiom of Choice) and the generalized continuum hypothesis $2^\mathfrak{c}=\mathfrak{c}^+$, where $\mathfrak{c}=\card(\mathbb{R})$).
\end{theorem}

\begin{proof}  We offer a model of the above system of our axioms in the framework of the (standard) analysis by constructing a differential algebra over a non-Archimedean field - we denote them by 
$\widehat{\mathcal{E}(\Omega)^\mathcal{D}}$, $\widehat{\mathbb{R}^\mathcal{D}}$ and $\widehat{\mathbb{C}^\mathcal{D}}$, respectively - which satisfy all of the above axioms if treated as $\widehat{\mathcal{E}(\Omega)}$, $\widehat{\mathbb{R}}$ and $\widehat{\mathbb{C}}$, respectively.  The construction of $\widehat{\mathcal{E}(\Omega)^\mathcal{D}}$ and $\widehat{\mathbb{C}^\mathcal{D}}$ appears in (\cite{TodVern08}, \S1-5). Here is a summary of this construction (warning: in this article we use the notation $\mathcal{D}=:\mathcal{D}(\mathbb{R}^d)$ instead of $\mathcal{D}_0$ in \cite{TodVern08}):
\begin{enumerate}
\item For any $\varphi\in\mathcal{D}=:\mathcal{D}(\mathbb{R}^d)$ we define the \textbf{radius of support} of $\varphi$ by
\begin{equation}\label{E: RadiusSupport}
R_\varphi=
\begin{cases}
 \sup\{||x|| : x\in\mathbb{R}^d,\;  \varphi(x)\not= 0 \}, &\text{\; if\; } \varphi\not=0,\\
1, 	&\text{\; if\; }\varphi=0. 
\end{cases}
\end{equation}

\item For any $n\in\mathbb{N}$ we define the \textbf{directing set} of test functions:
\begin{align}\notag
{\mathcal D}_{n}= \big\{&
\varphi\in{\mathcal D}: \varphi \text{\;is real-valued\,}, R_\varphi\leq 1/n, \quad \int_{\mathbb{R}^d}\varphi(x)\, dx=1,\\
&(\forall\alpha\in\mathbb{N}_0
^d)\big(1 \leq |\alpha|\leq n \Rightarrow\int_{\mathbb{R}^d}x^\alpha\varphi(x)\, dx=0\big),\notag\\ 
&\int_{\mathbb{R}^d}|\varphi(x)|\, dx \leq 1 + \frac{1}{n},\notag\\
&(\forall\alpha\in\mathbb{N}_0^d) (|\alpha|\leq n \Rightarrow\sup_{x\in\mathbb{R}^d}|{\partial^\alpha
\varphi(x)|}\leq (R_\varphi)^{-2(|\alpha| +d)})\notag\; \big\}.
\end{align}  
\item There exists a $\mathfrak{c}^+$-good free ultrafilter  
$\mathcal{U}$ on $\mathcal{D}$ such that $\mathcal{D}_n\in\mathcal{U}$ for all $n\in\mathbb{N}$ (\cite{TodVern08}, p. 210-213). 

\item Let $S$ be a set. 
 We denote by $S^\mathcal{D}$ the set of all \textbf{$\mathcal{D}$-nets in $S$}, i.e. all functions from $\mathcal{D}$ to $S$ (\cite{TodVern08}, p. 213-217). 

\item Let $(x_\varphi), (y_\varphi)\in \mathbb{R}^\mathcal{D}$ be to nets in 
$\mathbb{R}$. 
\begin{enumerate}
\item We say that $x_\varphi$ and $y_\varphi$ are \textbf{equal almost everywhere}, in symbol,  $x_\varphi\overset{\mathbf{\centerdot}}{=}y_\varphi$, if $\{\varphi\in\mathcal{D} :  x_\varphi=y_\varphi\}\in\mathcal{U}$.

\item We say that $x_\varphi$ \textbf{is less than} $y_\varphi$ \textbf{almost everywhere}, in symbol,  
$x_\varphi\overset{\mathbf{\centerdot}}{<} y_\varphi$, 
if $\{\varphi\in\mathcal{D} :  x_\varphi<y_\varphi\}\in\mathcal{U}$. Similarly, we say that $x_\varphi$ \textbf{is less or equal than} $y_\varphi$ \textbf{almost everywhere}, in symbol,  
$x_\varphi\overset{\mathbf{\centerdot}}{\leq} y_\varphi$, 
if $\{\varphi\in\mathcal{D} :  x_\varphi\leq y_\varphi\}\in\mathcal{U}$.

\item The terminology ``almost everywhere'' is justified by the following results: Let $p: \mathcal{P}(\mathcal{D})\to\{0, 1\}$ be defined by $p(S)=0$ if $S\notin\mathcal{U}$ and  $p(S)=1$ if $S\in\mathcal{U}$. Here $\mathcal{P}(\mathcal{D})$ stands for the power set of $\mathcal{D}$. Then $p$ is a \textbf{finitely additive probability measure} such that: {\rm (i)} $p(\mathcal{D}_n)=1$ for all $n\in\mathbb{N}$; {\rm (ii)} $p(S)=0$ for any finite set $S$. Also, $p(\cup_{n=1}^m S_n)=1$ implies $(\exists n\in\mathbb{N})(p(S_n)=1)$ (\cite{TodVern08}, p. 213).
\end{enumerate}
\item  We define the sets of the
\textbf{moderate} and \textbf{negligible} nets in $\mathbb{C}^{\mathcal{D}}$ by
\begin{align}
&\mathcal{M}(\mathbb{C}^{\mathcal{D}})=\big\{(z_\varphi)\in{\mathbb{C}^{\mathcal{D}}} : \;
(\exists m\in\mathbb{N})\big(|z_\varphi|\leqae (R_\varphi)^{-m} \big) \big\},\notag\\ 
&\mathcal{N}(\mathbb{C}^{\mathcal{D}})=\big\{(z_\varphi)\in{\mathbb{C}^{\mathcal{D}}}:\;
(\forall p\in\mathbb{N})\big(|z_\varphi|\lae (R_\varphi)^{p} \big)\big\}, \notag
\end{align}
respectively. The elements of  $\widehat{\mathbb{C}^\mathcal{D}}=
\mathcal{M}(\mathbb{C}^{\mathcal{D}})/\mathcal{N}(\mathbb{C}^{\mathcal{D}})$ are called \textbf{asymptotic numbers} and we
denote by
$\widehat{z_\varphi}\in\widehat{\mathbb{C}^\mathcal{D}}$ the equivalence class of the net $(z_\varphi)\in\mathbb{C}^\mathcal{D}$. We supply $\widehat{\mathbb{C}^\mathcal{D}}$ with the ring operations inherited from $\mathbb{C}^{\mathcal{D}}$. Also, we let $|\widehat{z_\varphi}|=\widehat{|z_\varphi|}$ for the \textbf{absolute value} of $\widehat{z_\varphi}$. 

\item We define a \textbf{non-Archimedean valuation} $v: \widehat{\mathbb{C}^\mathcal{D}}\to \mathbb{R}\cup\{\infty\}$ by $v(0)=\infty$ and  $v(\widehat{z_\varphi})=\sup\{r\in\mathbb{R}: r\leqae\ln|z_\varphi|/\ln{(R_\varphi)}\}$, if $\widehat{z_\varphi}\not=0$. We define the \textbf{ultra-norm} $|\cdot|_v: \widehat{\mathbb{C}^\mathcal{D}}\to \mathbb{R}$ by $|\widehat{z_\varphi}|_v=e^{-v(\widehat{z_\varphi})}$. Finally, we define the \textbf{ultra-metric} on $\widehat{\mathbb{C}^\mathcal{D}}$ by $d(a, b)_v=|a-b|_v$.

\item We define the \textbf{real asymptotic numbers} $\widehat{\mathbb{R}^\mathcal{D}}\subset\widehat{\mathbb{C}^\mathcal{D}}$ by $\widehat{x_\varphi}\in\widehat{\mathbb{R}^\mathcal{D}}$ if there exists a net $(y_\varphi)$ in $\mathbb{R}^\mathcal{D}$ such that  $x_\varphi\eqae y_\varphi$. We define an \textbf{order relation} on
$\widehat{\mathbb{R}^\mathcal{D}}$ as follows: Let
$\widehat{x_\varphi}\not= 0$. Then  
$\widehat{x_\varphi}>0$  if $x_\varphi\gae 0$. The asymptotic number $\widehat{\rho}=:\widehat{R_\varphi}$ is called the \textbf{canonical infinitesimal} in $\widehat{\mathbb{R}^\mathcal{D}}$. If $x\in\mathbb{R}$, we let $(\widehat{\rho})^x=\widehat{e^{xR_\varphi}}$.

\item We supply $\widehat{\mathbb{R}^\mathcal{D}}$ with the \textbf{order topology}. Also, we supply $\widehat{\mathbb{C}^\mathcal{D}}$ and $\widehat{\mathbb{R}^\mathcal{D}}^d=:\widehat{\mathbb{R}^\mathcal{D}}\times\widehat{\mathbb{R}^\mathcal{D}}\times\dots\times\widehat{\mathbb{R}^\mathcal{D}}$ with the corresponding product-order topology inherited from $\widehat{\mathbb{R}^\mathcal{D}}$.

\item\label{I: Embeddings} We define the \textbf{embeddings} $\mathbb{C}\subset\widehat{\mathbb{C}^\mathcal{D}}$ and 
$\mathbb{R}\subset\widehat{\mathbb{R}^\mathcal{D}}$ by the constant nets, i.e. by $z\to\widehat{z}$.

\item  Let $\Omega$ be an open set of
$\mathbb{R}^d$ and $R_\varphi$
be the radius of support of $\varphi$. A net $(f_\varphi)\in\mathcal{E}(\Omega)^{\mathcal{D}}$ is called \textbf{moderate} or \textbf{negligible}, if 
\begin{align}\label{E: rho-moderate}\notag
& (\forall K\Subset\Omega)(\forall\alpha\in\mathbb{N}_0^d)(\exists
m\in\mathbb{N})(\sup_{x\in K}|\partial^\alpha f_\varphi(x)|\leqae (R_\varphi)^{-m}), \notag\\
& (\forall K\Subset\Omega)(\forall\alpha\in\mathbb{N}_0^d)(\forall
p\in\mathbb{N})(\sup_{x\in K}|\partial^\alpha f_\varphi(x)|\lae (R_\varphi)^{p}),\notag
\end{align}
respectively. We denote by $\mathcal{M}(\mathcal{E}(\Omega)^{\mathcal{D}})$ and $\mathcal{N}(\mathcal{E}(\Omega)^{\mathcal{D}})$ the sets of the moderate and negligible $\mathcal{D}$-nets in $\mathcal{E}(\Omega)$, respectively. The elements of the factor ring $
\widehat{\mathcal{E}(\Omega)^\mathcal{D}}=
\mathcal{M}(\mathcal{E}(\Omega)^{\mathcal{D}})/\mathcal{N}(\mathcal{E}(\Omega)^
{\mathcal{D}})$ are called
\textbf{asymptotic functions} on $\Omega$. We denote by
$\widehat{f_\varphi}\in\widehat{\mathcal{E}(\Omega)^\mathcal{D}}$ the equivalence
class of the net $(f_\varphi)$.

\item\label{I: canonical embedding}  The \textbf{canonical embedding} $\sigma_\Omega: \mathcal{E}(\Omega)\to\widehat{\mathcal{E}(\Omega)^\mathcal{D}}$  is defined by $\sigma_\Omega(f)=\widehat{f}$. We also  define the \textbf{embedding} $\widehat{\mathbb{C}^\mathcal{D}}\subset\widehat{\mathcal{E}(\Omega)^{\mathcal{D}}}$ by 
$\widehat{z_\varphi}\to\widehat{Z_\varphi}$, where $Z_\varphi(x)=z_\varphi$ for all $\varphi\in\mathcal{D}$ and $x\in\Omega$.

\item\label{I: Embedding}  We define the \textbf{embedding} $E_\Omega: \mathcal{D}^\prime(\Omega)\to\widehat{\mathcal{E}(\Omega)^\mathcal{D}}$ by $E_\Omega(T)=\widehat{T\circledast\varphi}$, where $T\circledast\varphi$ stands for the ``generalized convolution'': convolution along with a cut off (see~\cite{TodVern08}, p. 214).

\item \label{I: Association} Let $\widehat{f_\varphi}\in\widehat{\mathcal{E}(\Omega)^\mathcal{D}}$ and $\tau\in\mathcal{D}(\Omega)$. We define the \textbf{pairing} $\big(\widehat{f_\varphi}\big|\tau\big)\in\widehat{\mathbb{C}^\mathcal{D}}$ by $\big(\widehat{f_\varphi}\big|\tau\big)=:\widehat{\left(f_\varphi|\tau\right)}$, where $\left(f_\varphi|\tau\right)=\int_\Omega f_\varphi(x)\tau(x)\, dx$. 

\item We let $\mu_\mathcal{D}(\Omega)=:\big\{\omega+h_\varphi: \omega\in\Omega, (h_\varphi)\in(\mathbb{R}^d)^\mathcal{D}, (\forall n\in\mathbb{N})(||h_\varphi||\lae1/n)\big\}$. We define an equivalence relation on $\mu_\mathcal{D}(\Omega)$ by $(x_\varphi)\sim(y_\varphi)$ if $||x_\varphi-y_\varphi||\in\mathcal{N}(\mathbb{C}^\mathcal{D})$ and we let let $\widehat{\mu_\mathcal{D}}(\Omega)=\mu_\mathcal{D}(\Omega)/\sim$ for the corresponding factor set. We denote by $\widehat{x_\varphi}\in\widehat{\mu_\mathcal{D}}(\Omega)$ the equivalence class of $(x_\varphi)\in\mu_\mathcal{D}(\Omega)$. Also, we denote by $\mathcal{C}^\infty(\widehat{\mu_\mathcal{D}}(\Omega), \widehat{\mathbb{C}^\mathcal{D}})$ the ring of the $\mathcal{C}^\infty$-functions from $\widehat{\mu_\mathcal{D}}(\Omega)$ to $ \widehat{\mathbb{C}^\mathcal{D}}$.

\item\label{I: Restriction} Let $\widehat{f_\varphi}\in\widehat{\mathcal{E}(\Omega)^\mathcal{D}}$ and $\mathcal{O}$ be an open subset of $\Omega$. We define the \textbf{restriction} $\widehat{f_\varphi}\rest\mathcal{O}\in \widehat{\mathcal{E}(\mathcal{O})^\mathcal{D}}$ of $\widehat{f_\varphi}$ on $\mathcal{O}$ by $\widehat{f_\varphi}\rest\mathcal{O}=:\widehat{f_\varphi\rest\mathcal{O}}$, where $f_\varphi\rest\mathcal{O}$ is the usual pointwise restriction of $f_\varphi$ on $\mathcal{O}$. 

\item\label{I: Values} Let $\widehat{f_\varphi}\in\widehat{\mathcal{E}(\Omega)^\mathcal{D}}$. The \textbf{graph} $\widehat{f_\varphi}: \widehat{\mu_\mathcal{D}}(\Omega)\to\widehat{\mathbb{C}^\mathcal{D}}$ is defined by $\widehat{f_\varphi}(\widehat{x_\varphi})=\widehat{f_\varphi(x_\varphi)}$. 
\end{enumerate}

We leave to the reader to verify that $\widehat{\mathbb{R}^\mathcal{D}}, \widehat{\mathbb{C}^\mathcal{D}}, \widehat{\mathcal{E}(\Omega)^\mathcal{D}},\, \widehat{\rho}, (\widehat{\rho})^x$ and $\widehat{\mu_\mathcal{D}}(\Omega)$ satisfy all axioms in Section~\ref{S: Generalized Scalars and Functions in Axioms} if treated as $\widehat{\mathbb{R}}, \widehat{\mathbb{C}}, \widehat{\mathcal{E}(\Omega)}, s, s^x$ and $\widehat{\mu}(\Omega)$, respectively. For more detailed discussion we refer to \cite{TodVern08}. We should note that the axiom $2^\mathfrak{c}=\mathfrak{c}^+$  (known as a {\em generalized continuum hypothesis}) is involved in the proof that $\widehat{\mathbb{R}^\mathcal{D}}$ is Cantor $\mathfrak{c}^+$-complete (\cite{TodVern08}, Theorem 6.3, p.227, and Corollary 7.5, p.229). 
\end{proof}

\section{Inconsistency in Laplace transform theory}
 
	Let $\mathcal{L}$ be the Laplace transform operator and $f(t)$ and $F(z)$ be two generalized functions, i.e. classical functions or Schwartz distributions. In this section we shall treat the equality $\mathcal{L}[f]=F$ as a statement, i.e. a predicate in two variables, $f$ and $F$, which is either true or false. Notice that in this article the active variables are $t$ and $z$ (not the more popular $t$ and $s$). The letter $s$ is preserved for the scale of the field $\widehat{\mathbb{R}}$ (Axiom 5), which means that $s$ is a (fixed) infinitesimal constant (and should be treated in a way similar to the way we treat $\pi$, $e$, etc.)

	The next example (which is rather a counterexample) shows that the formulas $\mathcal{L}[f^{\prime}]=z\mathcal{L}[f]-f(0)$, $\mathcal{L}[f^{\prime\prime}]=z^2\mathcal{L}[f]-z\,f(0)-f^\prime(0)$, $\mathcal{L}[\sin{t}]=\frac{1}{z^2+1}$ and $\mathcal{L}[\delta(t)]=1$, are inconsistent in Laplace transform theory. Consequently, the popular tables of Laplace transform formulas  (see, for example, \cite{jlSchiff99}, p. 209-218) - which include these four formulas - are also inconsistent.

\begin{example}(Counterexample)\label{Ex: Counterexample} We apply the Laplace operator $\mathcal{L}$  to the initial value problem $y^{\prime\prime}+y=\delta(t)$, $y(0_+)=0,\,  y^\prime(0_+)=1$ and the result is $(z^2+1)\mathcal{L}[y]=2$. The latter implies that $y=\mathcal{L}^{-1}\left[\frac{2}{z^2+1}\right]=2\sin{t}$ is a solution of this initial value problem. In particular, it follows that $y^\prime(0_+)=1$, i.e. $2=1$, a contradiction. Notice that the above initial value problem does have a solution, $y=H(t)\sin{(t)}$, in the space of Schwartz distributions $\mathcal{D}^\prime(\mathbb{R})$. That is to say that the statement $(\exists y\in\mathcal{D}^\prime(\mathbb{R}))(y^{\prime\prime}+y=\delta(t)$ and $y(0_+)=0$ and $y^\prime(0_+)=1)$ is true.
\end{example}

	We should notice that logical contradictions and inconsistency in the calculations of the type mentioned in the above example sometimes appear in the work of physicists and engineers - usually disguised behind a complicated terminology of the specific field. They (physicists and engineers) rarely blame mathematics and mathematicians...

	Here are three consistent versions of the Laplace transform theory each using part, but not all, of the formulas from a typical tables of Laplace transform formulas (see, for example, \cite{jlSchiff99}, p. 209-218). 

\begin{remark}(Laplace Transform in Engineer Mathematics). 
 One way to achieve a free of logical contradictions table of Laplace transform formulas while preserving the formulas such as $\mathcal{L}[f^{\prime}]=z\mathcal{L}[f]-f(0)$ and $\mathcal{L}[f^{\prime\prime}]=z^2\mathcal{L}[f]-z\,f(0)-f^\prime(0)$, etc., is to replace the formula $\mathcal{L}[\delta^{(n)}(t)]=z^n$, $n=0,1,2,\dots$, by the formulas $\mathcal{L}[\delta^{(n)}(t-\varepsilon)]= z^ne^{-\varepsilon z}$, $n=0, 1, 2,\dots$, where $\varepsilon\in\mathbb{R}_+$ (not $\varepsilon=0$). In other words, we define the Laplace transform in such a way that $\delta^{(n)}(t)\notin\dom(\mathcal{L})$ and $\delta^{(n)}(t-\varepsilon)\in\dom(\mathcal{L})$. In this approach the formula $\mathcal{L}[\delta(t)]=1$ should be treated as a ``trouble maker'' which should be ``expelled'' from the theory ``for good''. The presence of the formula $\mathcal{L}[\delta(t)]=1$ in mathematics is, perhaps, rooted, in the believe that $H(t)\delta(t)=\delta(t)$. Unfortunately, the product $H(t)\delta(t)$ does not exist in Schwartz theory of distributions. For a discussion we refer the reader to (\cite{jfCol84a}, Chapter 2) or (\cite{mGrosser at al 2}, \S 1.1) or (\cite{rpKanwal}, p. 209-210). This particular branch of Laplace transform theory is very popular among engineers who often treat the Dirac delta function at intuitive level - outside the framework of distribution theory (see Remark~\ref{R: Signal Analysis} at the end of this article). On the other hand, one might hope that the correct solution of the above initial value problem should be obtained by first, solving the initial value problem $y^{\prime\prime}+y=\delta(t-\varepsilon)$, $y(0_+)=0,\,  y^\prime(0_+)=1$ by the method of Laplace transform and then going to a weak limit as $\varepsilon\to 0_+$. Unfortunately, the result is the same as before: The Laplace transform produces $y_\varepsilon(t)=\sin{t}+H(t-\varepsilon)\sin{(t-\varepsilon)}$ and the weak limit produces, again, $y=2\sin{t}$ which does not satisfies $y^\prime(0_+)=1$. The last example, among others, indicates the limitations on this particular branch of Laplace transform theory. 
\end{remark}
\begin{remark}(Laplace Transform in Distribution Theory). In Schwartz theory of distributions the formulas such as $\mathcal{L}[f^{\prime}]=z\mathcal{L}[f]-f(0)$ and $\mathcal{L}[f^{\prime\prime}]=z^2\mathcal{L}[f]-z\,f(0)-f^\prime(0)$ do not have chance to survive since the Schwartz distributions do not have, in general, pointwise values. For example, a formula such as $\mathcal{L}[\delta^{\prime}]=z\mathcal{L}[\delta]-\delta(0)$  does not make sense, since $\delta(0)$ is not well-defined.  Instead, the Laplace transform of a distribution $f\in\mathcal{D}^\prime(\mathbb{R})$ is defined by the formula $\mathcal{L}[f](z)=\mathcal{F}[f](-iz)$, where $\mathcal{F}$ stands for {\em Fourier transform operator}. Here $f\in\dom(\mathcal{L})$ and $z\in \dom(\mathcal{L}[f])$, where $\dom(\mathcal{L})=\{f\in\mathcal{D}^\prime(\mathbb{R}): \supp(f)\subseteq[0, \infty), (\exists a\in\mathbb{R})(e^{-at}f(t)\in\mathcal{S}^\prime(\mathbb{R})\}$ and $ \dom(\mathcal{L}[f])=\{z\in\mathbb{C}: {\rm Re}(z)>\alpha_f \}$, where $\alpha_f=\inf\{a\in\mathbb{R}: e^{-at}\, f(t)\in\mathcal{S}^\prime(\mathbb{R})\}$.  This definition of Laplace transform leads to the formulas $\mathcal{L}[f^{(n)}]=z^n\mathcal{L}[f]$, $n=0, 1, 2,\dots$ for all distributions $f\in\dom(\mathcal{L})$. This explains the origin of the formulas $\mathcal{L}[\delta]=1$ and $\mathcal{L}[\delta^{(n)}]=z^n$, since $\mathcal{F}[\delta]=1$. This version of the Laplace transform theory is logically consistent, because Fourier transform theory is consistent in the framework of distribution theory. However, formulas such as $\mathcal{L}[f^{\prime}]=z\mathcal{L}[f]-f(0_+)$ are not part of this theory regardless whether or not the value $f(0_+)$ exists (except, of course, in the case $f(0_+)=0$). This version of Laplace transform theory admits generalization to distributions in many variables. It should be viewed as a particular case of Fourier transform theory. For more details we refer to (\cite{vVladimirov}, p. 143-151). 
\end{remark}

\begin{remark}(Laplace Transform in Colombeau Theory). The particular brand of Laplace transform theory described above (as a particular case of Fourier transform) was successfully extended to the algebra of Colombeau's tempered generalized functions in \cite{KarunaVijayan93} and in \cite{NedelkovPilipovic95}. Formulas involving pointwise values however, such as $\mathcal{L}[f^{\prime}]=z\mathcal{L}[f]-f(0)$, $\mathcal{L}[f^{\prime\prime}]=z^2\mathcal{L}[f]-z\,f(0)-f^\prime(0)$ and $\mathcal{L}[\delta^{\prime}]=z\mathcal{L}[\delta]-\delta(0)$, are not part of this generalization. Notice that - unlike in distribution theory - the values $f(0), f^\prime(0), \delta(0), \delta^{\prime}(0)$, etc. do make sense in Colombeau theory since Colombeau's generalized functions have pointwise values in the ring of Colombeau's generalized numbers (\cite{jCol85}, \S 2.1). As far as we know however, there have not been attempt so far to reconcile Laplace transform theory involving Schwartz distributions with the pointwise values of generalized functions in the framework of Colombeau's theory. 
\end{remark}
\section{Laplace transform in $\widehat{\mathcal{E}(\Omega)}$}
In this section we show the Laplace transform theory is free of contradictions in the framework of the algebra of generalized functions $\widehat{\mathcal{E}(\mathbb{R})}$. Notice that every generalized function $f\in\widehat{\mathcal{E}(\mathbb{R})}$ is a mapping of the form $f: \widehat{\mu}(\mathbb{R})\to \widehat{\mathbb{C}}$ (Axiom 8). That means that all Schwartz distributions in $\mathcal{D}^\prime(\mathbb{R})$, if embedded into $\widehat{\mathcal{E}(\mathbb{R})}$ (Axiom 9), are also mapping of the same type. In particular, the values $f(0), f^\prime(0), \delta(0), \delta^{\prime}(0)$, etc. are always well defined. What follows is not a comprehensive theory (which will be a topic for another article); rather we shall here restrict our discussion only to those tempered generalized functions which belong to $\widehat{\mathcal{E}(\Omega)}$ and which can be treated in the framework of our axiomatic approach. 

	In what follows we denote by $E[0, \infty)$ the set of all functions $f:[0, \infty)\to\mathbb{C}$ whose restriction on $\mathbb{R}_+$ belongs to $\mathcal{L}_{loc}(\mathbb{R}_+)$ with exponential growth at infinity (\cite{gbFolland}, p. 256-257). 

\begin{definition} \begin{enumerate} 
\item We denote by $\dom(\widehat{\mathcal{L}})$ the set of all functions $f\in \widehat{\mathcal{E}(\mathbb{R})}$ which can be presented in the form $f=\sum_{n=1}^\nu \alpha_n\phi_n+\sum_{m=1}^\mu \beta_m \psi_m$ for some $\nu, \mu\in\mathbb{N}$, $\alpha_n, \beta_m\in\widehat{\mathbb{C}}$, $\phi_n\in E[0, \infty)$ and $\psi_m\in\widehat{\mathcal{E}(\mathbb{R})}$, where $\psi_m$ has an external support $\supp(\psi_m)$ which is a compact subset of $[0, \infty)$ and an internal support ${\rm Supp}(\psi_m)$ which is a subset of $[s, \infty)$ (Definition~\ref{D: External and Internal Support}). Here  $s$ stands for the scale of $\widehat{\mathbb{R}}$ (Axiom 5) and $[s, \infty)$ is a short notation for $\{t\in\widehat{\mathbb{R}}: t\geq s\}$.

	\item Let $f\in\dom(\widehat{\mathcal{L}})$.  For every finite $z\in\widehat{\mathbb{C}}$ with a sufficiently large ${\rm Re}(z)$, we define the  Laplace transform $\widehat{\mathcal{L}}(f)(z)$ of $f$  by linearity: (a) $\widehat{\mathcal{L}}(\phi_n)(z)=\sigma_{(\lambda, \infty)}\left(\int_0^\infty \phi_n(t)e^{-zt}\, dt\right)$, where $\int_0^\infty \phi_n(t)e^{-zt}\, dt$ is the usual (classical) Laplace transform of $\phi_n$, defined for all $z\in\mathbb{C}$, such that  ${\rm Re}(z)>\lambda$, for a sufficiently large $\lambda \in\mathbb{R}_+$, and $\sigma_{(\lambda, \infty)}$ is the embedding mentioned in Axiom 7; (b) the integral in $\widehat{\mathcal{L}}(\psi_m)(z)=\int_{-\infty}^\infty \psi_m(t)e^{-zt}\, dt$ (defined for all finite $z\in\widehat{\mathbb{C}}$)  is in the sense of Definition~\ref{D: Integral} and $e^{-zt}$ is a short notation for $\sigma_{\mathbb{R}^3}(e^{-zt})$ (Example~\ref{Ex: Exponents}).
\end{enumerate}
\end{definition}

	We leave the proof of the next result to the reader.
\begin{lemma}(Inverse Laplace).\label{L: Inverse Laplace} Let $f, g\in \dom(\widehat{\mathcal{L}})$. Then $f\cong g$ \ifff $\widehat{\mathcal{L}}(f)=\widehat{\mathcal{L}}(g)$.
\end{lemma}

\begin{example} \begin{enumerate}\item Let $\delta\in\mathcal{D}^\prime(\Omega)$ denote the Dirac distribution supported at  $\{0\}$ and let $E_\mathbb{R}(\delta)$ be its image in $\widehat{\mathcal{E}(\mathbb{R})}$ under the embedding  $E_\mathbb{R}$. We shall denote this image again by $\delta$. We have $\supp(\delta)=\{0\}$ by Lemma~\ref{L: Preservation of External Support} . However, $\delta\notin\dom(\widehat{\mathcal{L}})$ because ${\rm Supp}(\delta)$ is not a subset of $[s, \infty)$ (Axiom 9, (e)). Thus $\widehat{\mathcal{L}}(\delta)$ is undefined. 
\item In contrast to the above, let $\delta(t-2s)$ stand for $E_\mathbb{R}(\delta)(t-2s)$ (Example~\ref{Ex: Shifted Delta}).  We have $\supp(\delta(t-2s))=\{0\}$ and also ${\rm Supp}(\delta(t-2s))\subseteq [s, 3s]$ which is a subset of $[s, \infty]$. Thus  $\delta(t-2s)\in\dom(\widehat{\mathcal{L}})$.
By direct calculation we derive the formula $\widehat{\mathcal{L}}(\delta(t-2s))=e^{-2sz}$. Similarly, $\widehat{\mathcal{L}}(\delta^{(n)}(t-2s))=z^ne^{-2sz}$.  The function $\delta(t-2s)$ is associated with $\delta$ in the sense that $\int_{-\infty}^\infty \delta(t-2s)\tau(t)\, dt\approx \tau (0)$ for all test functions $\tau\in\mathcal{E}(\mathbb{R})$. We also have $\int_{-\infty}^\infty \delta(t-2s)\, dt=1$.

\item Consider the initial value problem $y^{\prime\prime}+y\cong\delta(t-2s)$, $y(0)=0,  y^\prime(0)=1$ similar, but different from those in Example~\ref{Ex: Counterexample}. Notice that the equality, $=$, in Example~\ref{Ex: Counterexample} has been replaced by the weak inequality, $\cong$ (Definition~\ref{D: Weak Equalities}). We apply the operator $\widehat{\mathcal{L}}$ and (with the help of Lemma~\ref{L: Inverse Laplace}) we obtain $(z^2+1)\widehat{\mathcal{L}}[y]-1= e^{-2sz}$ which implies  $y\cong\widehat{\mathcal{L}}^{-1}\left[\frac{1}{z^2+1}\right] +\widehat{\mathcal{L}}^{-1}\left[\frac{e^{-2sz}}{z^2+1}\right]=\sin{t}+H(t-2s)\sin{(t-2s)}$. The direct calculations show that this generalized function is indeed a solution of the above initial value problem. 

\item Here is another initial value problem: $y^{\prime\prime}+y=\delta(t)$,\; $y(0)=0,\,  y^\prime(0)=0$. Notice that this initial value problem admits a distributional solution $y=\left(H(t)-1\right)\sin{t}$. However, the classical Laplace transform (and the usual tables of Laplace transform formulas) leads again to a contradiction. Indeed, $(z^2+1)\mathcal{L}[y]=1$ implies that $y=\mathcal{L}^{-1}\left[\frac{1}{z^2+1}\right]=\sin{t}$ satisfies the above initial value problem. In particular, $y^\prime(0)=0$ leading to $1=0$, a contradiction. In contrast to the above, let us apply the generalized Laplace operator $\widehat{\mathcal{L}}$ to  the initial value problem $y^{\prime\prime}+y\cong\delta(t-2s)$,\; $y(0)=0,\,  y^\prime(0)=0$.  The result is $y\cong H(t-2s)\sin{(t-2s)}$ which is indeed the solution we are looking for.  
\end{enumerate}  
\end{example}

\begin{remark}(Signal Analysis)\label{R: Signal Analysis}. The delta function $\delta(t-2s)$ - which appears in this article - should be viewed as a theoretical idealization of the delta-net $(\delta_\varepsilon):\mathbb{R}\to\mathbb{R}$, defined by
\begin{equation}
\delta_\varepsilon(t)=
\begin{cases}
 1/2\varepsilon, &\text{\; if\; } 0\leq t<2\varepsilon,\\
0, 	&\text{\; otherwise},
\end{cases}
\end{equation}
where $\varepsilon$ is a ``small parameter''. Notice that this is not a traditional delta-net; it is rather delta-net shifted to the right at a distance $\varepsilon$. Such ``shifted delta-nets'' (instead of $\delta(t-2s)$) appear in the inverse problem in signal analysis, when we sometimes try to find the solution of the initial value problem $Ly^{\prime\prime}+ Ry^\prime+\frac{1}{C}y=f(t)$,\; $y(0_+)=y^\prime(0_+)=0$, without knowing the values of the inductance $L$, resistance $R$ and capacitance $C$ of an electrical circuit. Here $f(t)=E^\prime(t)$, where $E(t)$ stands for the impressed voltage. In these circumstances we often try to find an approximate solution $y_\varepsilon$ of 
$Ly^{\prime\prime}+ Ry^\prime+\frac{1}{C}y=\delta_\varepsilon(t)$,\; $y(0_+)=y^\prime(0_+)=0$ by a physical experiment and then apply convolution $y=y_\varepsilon\star f$. The impressed voltage signal in such physical experiments must be generated of the form
\begin{equation}
E_\varepsilon(t)=
\begin{cases}
0, 	&\text{\; if\; } t< 0,\\
 t/2\varepsilon, &\text{\; if\; } 0\leq t<2\varepsilon,\\
1 &\text{\; if\; } t\geq 2\varepsilon,
\end{cases}
\end{equation}
which is different from the usual Heaviside function $H(t)$. Rather, $E_\varepsilon(t)$ is a ``shifted to the right'' version of $H(t)$.
\end{remark}
\begin{remark}(Laplace Transform Formulas with an Infinitesimal Constant). One way to obtain a table of Laplace transforms free of logical contradiction, is to replace the formulas: $\mathcal{L}(\delta^{(n)})(z)=z^n$ by $\widehat{\mathcal{L}}(\delta^{(n)}(t-2s))=z^ne^{-2sz}$ ($n=0,1,2,\dots$), respectively, where $s$ is a positive infinitesimal constant (which should be treated in a way we treat $\pi$, $e$, etc.). This, of course, will not make mathematical sense unless we also change the framework of the theory as well: that means to replace the spaces of functions $\mathcal{D}^\prime(\mathbb{R})$ and the fields of scalars $\mathbb{R}$ and $\mathbb{C}$ by $\widehat{\mathcal{E}(\mathbb{R})}$, $\widehat{\mathbb{R}}$ and $\widehat{\mathbb{C}}$, respectively. Is all these worth the efforts ? The answer very much depends on the scope of  applications we have in mind. But in any case we believe that the first and most important goals of any mathematical field - with priority over everything else - is to be free of logical contradictions. 
\end{remark}
{\bf Acknowledgment:} {\small The author thanks the anonymous referees for the numerous constructive remarks which improved the quality of the manuscript. The author also thanks to G\"{u}nther H\"{o}rmann for the discussion of some topics of the manuscript.}

\end{document}